\documentclass[12pt,oneside,reqno]{amsart}
\usepackage{amsmath,amssymb,amsfonts,graphics}
\usepackage{amsopn}
\usepackage{amsthm}
\usepackage{todonotes}

\newcommand{\F}{{\mathbb F}}

\theoremstyle{plain}
\newtheorem{thm}{Theorem}
\newtheorem{cor}[thm]{Corollary}
\newtheorem{lemma}[thm]{Lemma}
\newtheorem{prop}[thm]{Proposition}

\theoremstyle{definition}
\newtheorem{defn}{Definition}

\theoremstyle{remark}
\newtheorem{remark}{Remark}

\input{pdfcolor}

\begin{document}
	
\title{Graphs of Vectorial Plateaued Functions as  Difference Sets}
\author[Cesmelioglu]{Ay\c ca \c Ce\c smelio\u glu}
\author[Olmez]{Oktay Olmez}
\thanks{Oktay Olmez's research was supported by TUBITAK Research Grant Proj. No. 115F064.} 
\address{Department of Mathematics, Faculty of Science, Ankara University,
	Tandogan, Ankara, 06100, Turkey.} \email[O. ~Olmez]{oolmez@ankara.edu.tr}

	\maketitle
	
	\begin{abstract}
	
A function $F:\mathbb{F}_{p^n}\rightarrow \mathbb{F}_{p^m},$ is a vectorial $s$-plateaued function if for each component function $F_{b}(\mu)=Tr_n(\alpha F(x)), b\in \mathbb{F}_{p^m}^*$ and $\mu \in \mathbb{F}_{p^n}$, the Walsh transform value $|\widehat{F_{b}}(\mu)|$ is either $0$ or  $ p^{\frac{n+s}{2}}$. In this paper, we explore the relation between (vectorial) $s$-plateaued functions and partial geometric difference sets. Moreover, we establish the link between three-valued cross-correlation of $p$-ary sequences and vectorial $s$-plateaued functions. Using this link, we provide a partition of $\mathbb{F}_{3^n}$ into partial geometric difference sets. Conversely, using a partition of $\mathbb{F}_{3^n}$ into partial geometric difference sets, we constructed ternary plateaued functions $f:\mathbb{F}_{3^n}\rightarrow \mathbb{F}_3$. 
We also give a characterization of $p$-ary plateaued functions in terms of special matrices which enables us to give the link between such functions and second-order derivatives using a different approach.

%******************

	%	In this paper, we investigate image sets of some special functions known as plateaued functions. We showed that graph of a plateaued function is a partial geometric difference set. By using this characterization we provide a construction of a family of vectorial plateaued functions related to three-valued cross-correlation functions. We also showed when the characteristic $p=3$, a certain decimation provides a disjoint partition of the finite field into partial geometric difference sets. We also provide characterization of plateaued functions in terms of properties of partial geometric difference sets.

	\end{abstract}

\section{Introduction}
A (block) design is a pair $(\mathcal{P}, \mathcal{B})$ consisting of a finite set $\mathcal{P}$ of points and a finite collection $\mathcal{B}$ of nonempty subsets of $\mathcal{P}$ called blocks. Designs serve as a fundamental tool to investigate combinatorial objects. Also designs have attracted many researchers from different fields for solutions of applications problems including binary sequences with $2$-level autocorrelation, optical orthogonal codes, low density parity check codes, synchronization, radar, coded aperture imaging, and optical image alignment, distributed storage systems and cryptographic functions with high nonlinearity\cite {chung1989,delsarte,dillon1974,ding2015,olmez2}.   

One of the main construction method of designs is called difference set method. This method served as a powerful tool to construct symmetric designs, error correcting codes, graphs and cryptographic functions \cite  {assmus1992,bernasconi2001,bjl,brouwer2012,buratti2011,dingbook,dingbook2,pott,pott2011}. This paper will focus on the links between designs and a family of a function known as plateaued functions from cryptography. Especially we will investigate the connections between partial geometric difference sets and graph of plateaued functions.

%For a prime $p$, we define a primitive complex $p$-th root of unity $\zeta_p=e^{\frac{2\pi i}{p}}$. 
A function from the field  $\mathbb{F}_{p^{n}}$ to $\mathbb{F}_p$ is called a $p$-ary function. If $p=2$ then the function is simply called as Boolean. $p$-ary functions with various characteristics have been an active research subject in cryptography. Bent functions and plateaued functions are two well-known families which has prominent properties in this field \cite{carlet2003,CCC,Carlet,Carlet1,Carlet2,CarM,Mes1,Mes2}. These two families of functions can be characterized by their Walsh spectrum. A function $f$ from $\mathbb{F}_{p^{n}}$ to $\mathbb{F}_p$ is called an $s$-plateaued function if the Walsh transform $|\widehat{f}(\mu)|\in \{0,p^{\frac{n+s}{2}}\}$ for each $\mu \in\mathbb{F}_{p^n}$. A $0$-plateaued function $f$ is called as bent and its Walsh transform satisfies $|\widehat{f}(\mu)|= p^{\frac{n}{2}}$ for each $\mu \in\mathbb{F}_{p^n}$. Plateaued functions and bent functions play a significant role in cryptography, coding theory and sequences for communications \cite{Carlet1, Carlet2, CarM}. 

Boolean bent functions were introduced by Rothaus in \cite{rot1976}. These functions have optimal nonlinearity and can only exist when $n$ is even. In \cite{dillon1974}, it is shown that the existence of Boolean bent functions is equivalent to the existence of a family of difference sets known as Hadamard difference sets. Boolean plateaued functions are introduced by Zheng and Zhang as a generalization of bent functions in \cite{ZZ}. Boolean plateaued functions have attracted the attention of researchers since these functions provide some suitable candidates that can be used in cryptosystems. A difference set characterization of these functions was recently provided by the second author. In \cite{olmez2}, it is shown that the existence of Boolean plateaued functions is equivalent to the existence of partial geometric difference sets. 

In arbitrary characteristic, the graph of $f:\mathbb{F}_{p^n}\rightarrow \mathbb{F}_{p^m}$, $G_{f}=\{(x,f(x)):x\in \mathbb{F}_{p^{n}}\}$, plays an important role for the relation to difference sets, \cite{pott1, tan2010}. For instance, the graph of a $p$-ary bent function can be recognized as a relative difference set. In general, a characterization of plateaued functions in terms of difference sets is not known. A partial result in this direction is provided in \cite{cmt} for partially bent functions which is a subfamily of plateaued functions. 

There are recent result concerning explicit characterization of plateaued functions in odd characteristics through their second order derivatives in \cite{CMOS,mos1,mos2}.

In this paper, we first investigate the link between the graph of a plateaued function and partial geometric difference set. We also provide several characterization of plateaued functions in terms of associated difference set properties. By using these characterization we provide a family of vectorial plateaued functions which has an interesting connection to three-valued cross correlation functions.

The organization of the paper is as follows. In Section 2, we provide preliminary results concerning partial geometric difference sets. In Section 3, we mainly provide the links between vectorial plateaued functions and partial geometric difference sets. We also provide a construction as a result of our characterizations.
In Section 4, we focus on $p-$ary plateaued function. We provide several characteristics which are obtained from Butson-Hadamard-like matrices. This section also provides results concerning partially bent functions and partial geometric 
designs.

\section{Preliminaries}
Let $G$ be a group of order $v$ and let $S \subset G$ be a $k$-subset. For each $g \in G$, we define \[\delta(g):=|\{(s,t)\in S \times S \colon g=st^{-1}\}|.\] 
Next we define the difference sets of our interest.

\begin{defn} Let $v, k$ be positive integers with $v>k>2$.
	Let $G$ be a group of order $v$.  A $k$-subset $S$ of $G$ is called a partial geometric difference set (PGDS) in $G$ with parameters $(v, k; \alpha,\beta)$ if there exist constants $\alpha$ and $\beta$ such that, for each $x\in G$,
	\[\sum\limits_{y\in S}\delta(xy^{-1})=\left \{\begin{array}{ll} \alpha & \mbox{if } x\notin S,\\
	\beta & \mbox{if } x\in S.\\ \end{array} \right .\]
\end{defn}

There are two subclasses of PGDS namely difference sets and semiregular relative difference sets which have deep connections with coding theory, and cryptography \cite{assmus1992, dingbook,dingbook2}:
\begin{itemize}
	\item A $(v,k,\lambda)$-difference set (DS) in a finite group $G$ of order $v$ is a $k$-subset $D$ with the property that $\delta(g)=\lambda$ for all nonzero elements of $G$.
	
	\item A $(m,u,k,\lambda)$-relative difference set (RDS) in a finite group $G$ of order $m$ relative to a (forbidden) subgroup $U$ is a $k$-subset $R$ with the property that 
\[ \delta(g)=\begin{cases} 
k & g=1_G \\
\lambda & g\in G \backslash U\\
0 & otherwise \\
\end{cases}
\]
The RDS is called {\em semiregular} if $m = k = u \lambda$. 
	
\end{itemize}

Clearly  a $(v,k,\lambda)$-DS is a $(v,k; k\lambda, n+k\lambda)$-PGDS and an $(m,u,k,\lambda)$ semiregular RDS is a $(mu,k; \lambda(k-1), k(\lambda+1)-\lambda)$-PGDS \cite{olmez1}.

Group characters are powerful objects to investigate various types of difference sets.
A {\em character} $\chi$ of an abelian group $G$ is a homomorphism from $G$ to the multiplicative group of the complex numbers. The character $\chi_0$ defined by $\chi_0(g) = 1$ for all $g \in G$ is called the {\em principal character}; all other characters are called {\em nonprincipal}. We define the character sum of a subset $S$ of an abelian group $G$ as $\chi(S):= \sum_{s \in S} \chi(s)$. 

\begin{thm}[Theorem 2.12 \cite{olmez1}]
	\label{chartheoryPGDS}
	A $k$-subset $S$ of an abelian group $G$ is a partial geometric difference set in $G$ with parameters $(v,k;\alpha,\beta)$ if and only if $|\chi (S)|=\sqrt{\beta-\alpha}$ or $\chi (S)=0$ for every non-principal character $\chi$ of $G$.
\end{thm}

For instance, let $f$ be a $p$-ary bent function from the field  $\mathbb{F}_{p^{n}}$ to $\mathbb{F}_{p}$.  The set $G_f=\{(x,f(x)): x\in \mathbb{F}_{p^{n}} \}$ is called graph of $f$. Any non-principal character $\chi$ of the  additive group of $\mathbb{F}_{p^{n}}\times \mathbb{F}_p$ satisfies $|\chi(G_f)|^2=p^n$ or $0$. This observation yields that $G_f$ is a  $(p^{n+1},p^n, p^{2n-1}-p^{n-1},p^{2n-1}-p^{n-1}+p^n)$-PGDS in $H=\mathbb{F}_{p^{n}}\times\mathbb{F}_{p}$. 

 Walsh transform provides interesting connections between $p$-ary functions and difference sets. For a prime $p$, we define a primitive complex $p$-th root of unity $\zeta_p=e^{\frac{2\pi i}{p}}$. Let $f$ be a function from the field  $\mathbb{F}_{p^{n}}$ to $\mathbb{F}_p$ and let $F(x)=\zeta_p^{f(x)}$.  The Walsh transform of $f$ is defined as follows:
 $$\displaystyle \widehat{f}(\mu)=\sum_{x\in \mathbb{F}_{p^{n}}}  \zeta_p^{f(x)-Tr_n(\mu x)},~~~\mu \in \mathbb{F}_{p^{n}}$$ where $$Tr_n(z)=\sum_{i=0}^{n-1}z^{p^i}.$$  
 
 The convolution of $F$ and $G$ is defined by 
 $$(F * G)(a)=\sum_{x\in \mathbb{F}_{p^n} } F(x)G(x-a).$$ We will also take advantage of the convolution theorem of Fourier analysis. The Fourier transform of a
 convolution of two functions is 
 $$\widehat{F * G}=\widehat{F}\cdot \widehat{G}.$$

 \section{Results on Vectorial Functions}
 
 Let $F$ be a vectorial function from $\mathbb{F}_{p^{n}}$  to $\mathbb{F}_{p^{m}}$. For every $b\in \mathbb{F}^*_{p^{m}}$, the component function $F_b$ of $F$ from $\mathbb{F}_{p^{n}}$ to $\mathbb{F}_p$ is defined as $F_{b}(x)=Tr_m(bF(x))$. A vectorial function is called \textit{vectorial plateaued} if all its nonzero component functions are plateaued. If the nonzero component functions of a vectorial plateaued function are s-plateaued for the same $0 \leq s \leq n$ then $F$ is called as \textit{s-plateaued} following the terminology in \cite{mos1}.
 
 The set $G_{F}=\{(x,F(x)):x\in \mathbb{F}_{p^{n}}\}$ is called the graph of $F$. Next we will characterize vectorial functions by their graphs.
 
 \begin{thm} \label{plateaued-PGDS}
 	Let $F:\mathbb{F}_{p^{n}} \rightarrow \mathbb{F}_{p^{m}}$ be a vectorial function. Then the graph of $F$ is a $(p^{n+m},p^n; \alpha,\beta)$ partial geometric difference set  in $H=\mathbb{F}_{p^{n}}\times\mathbb{F}_{p^m}$ satisfying $\beta-\alpha=\theta$ if and only if $|\widehat{F_b}(a)| \in \{0, \sqrt{\theta}\}$ for all non zero $b \in \mathbb{F}_{p^{m}}$ and $a \in \mathbb{F}_{p^{n}}$. In particular, $\alpha=p^{2n-m}-p^{n+s-m}$ and $\beta=p^{n+s}+p^{2n-m}-p^{n+s-m}$.
 \end{thm}
 
 \begin{proof}
 	A non-principal character of $\mathbb{F}_{p^{n}}\times \mathbb{F}_{p^{m}}$ can be written as $\chi_{(a,b)}(x,y)=\zeta_p^{Tr_{n}(ax)+Tr_m(by)}$ for a nonzero $(a,b) \in \mathbb{F}_{p^{n}}\times \mathbb{F}_{p^{m}}$.
 	For any nonzero $b \in \F_{p^n}$, the Walsh transform of $F_b$ is
 	\begin{align*}
 	\widehat{F_b}(a)=&\sum_{x\in \mathbb{F}_{p^{n}}}\zeta_p^{-Tr_{n}(ax)+Tr_m(bF(x))}=\sum_{x\in \mathbb{F}_{p^{n}}}\chi_{(-a,b)}(x,F_b(x)) \\
 	=&\chi_{(-a,b)}(G_F)
 	\end{align*}
 	for any $a \in F_{p^m}$. 
 	Therefore $|\widehat{F_b}(a)| \in \{0, \sqrt{\theta}\} |$ for all non zero $b \in \mathbb{F}_{p^{m}}$ if and only if $G_{F}=\{(x,F(x)):x\in \mathbb{F}_{p^{n}}\}$ is a $(p^{n+m},p^n;\alpha,\beta)$ partial geometric difference set satisfying $\beta-\alpha=\theta$. Using the well-known \textit{Parseval identiy}, one immediately sees that $\theta=p^{n+s}$ and $|\{a\in \mathbb{F}_{p^{n}}:\widehat{F_b}(a)\neq 0\}|=p^{n-s}$ for some $0 \leq s \leq n$.
 	The parameters of a partial geometric difference set satisfies the relation in \cite{olmez1} and hence we have
 	\[p^{3n}=(\beta-\alpha)p^n+\alpha\nu=p^{n+s}p^n+\alpha\nu.\]
 	Then we see that $\alpha=p^{2n-m}-p^{n+s-m}$ and $\beta=p^{n+s}+p^{2n-m}-p^{n+s-m}$.
 \end{proof}
 \begin{remark} 
 	Theorem \ref{plateaued-PGDS} implies that a vectorial function $F:\mathbb{F}_{p^{n}} \rightarrow \mathbb{F}_{p^{m}}$ is s-plateaued if and only if its graph is a partial geometric difference set with the parameters $(p^{n+m},p^n;p^{2n-m}-p^{n+s-m},p^{n+s}+p^{2n-m}-p^{n+s-m})$. Note that Theorem \ref{plateaued-PGDS}  is also valid for $m=1$, i.e. the case of p-ary functions. Since for an $s$-plateaued $p$-ary function $f:\mathbb{F}_{p^{n}}\rightarrow \mathbb{F}_{p}$, the function $bf(x)$ is $s$-plateaued for each $b\in \mathbb{F}_{p}^*$, we can consider $f$ as a vectorial $s$-plateaued function.  
 	The case $s=0$ is the case of vectorial bent functions and if we additionally have $m=n$, these vectorial functions are known as \textit{planar} functions \cite{ColM}.
 \end{remark}
 
 Next we will investigate links between vectorial $s-$plateaued functions and partial geometric difference sets.
 
 \begin{prop} \label{diff-plateauedprop}
 	Let $F:\mathbb{F}_{p^{n}}\rightarrow \mathbb{F}_{p^{m}}$ be a vectorial function. Then $F$ is $s$-plateaued if and only if
 	\[\sum\limits_{a \in \mathbb{F}_{p^{n}}}|\{s \in \mathbb{F}_{p^{n}}\colon  y=F(s+x-a)-F(s)+F(a)\}|=\left \{\begin{array}{ll} \alpha & \mbox{if } y\neq F(x),\\
 	\beta & \mbox{if } y=F(x)\\ \end{array} \right . \]
 	
 \end{prop}
 
 \begin{proof}
 	
 	\[\begin{split}
 	\delta((x,y))&=|\{((s_1,t_1),(s_2,t_2))\in G_F \times G_F \colon x=s_1-s_2, y=t_1-t_2=F(s_1)-F(s_2)\}|\\
 	&=|\{s_2 \in \mathbb{F}_{p^{n}}\colon  y=F(s_2+x)-F(s_2)\}|\\
 	\end{split}\]
 	So the criteria for PGDS is given by 
 	\[\sum\limits_{a \in \mathbb{F}_{p^{n}}}\delta((x-a,y-F(a)))=\left \{\begin{array}{ll} \alpha & \mbox{if } y\neq F(x),\\
 	\beta & \mbox{if } y=F(x)\\ \end{array} \right .\]
 	and hence
 	\[\sum\limits_{a \in \mathbb{F}_{p^{n}}}|\{s \in \mathbb{F}_{p^{n}}\colon  y=F(s+x-a)-F(s)+F(a)\}|=\left \{\begin{array}{ll} \alpha & \mbox{if } y\neq F(x),\\
 	\beta & \mbox{if } y=F(x)\\ \end{array} \right .\]
 \end{proof}

 The above result can be associated with the derivative of an $s-$plateaued function. The derivative of a vectorial function is defined by 
 $$D_aF(x)=F(x+a)-F(x).$$
 
 To see the connection let us first replace $y$ in the expression
 \[ y=F(s+x-a)-F(s)+F(a)\]
 by $F(x)-c$ for $c \in \mathbb{F}_{p^n}$. Hence we have
 \[c=F(s)-F(a)-F(s+x-a)+F(x)=D_{s-a}F(a)-D_{s-a}F(x).\]
 This observation yields
 \begin{align*}
 &\sum\limits_{a \in \mathbb{F}_{p^{n}}}|\{s \in \mathbb{F}_{p^{n}}\colon  y=F(s+x-a)-F(s)+F(a)\}|\\
 =&\sum\limits_{a \in \mathbb{F}_{p^{n}}}|\{s \in \mathbb{F}_{p^{n}}\colon  D_{s-a}F(a)-D_{s-a}F(x)=c\}|\\
 =&\sum\limits_{a \in \mathbb{F}_{p^{n}}}|\{t \in \mathbb{F}_{p^{n}}\colon  D_tF(a)-D_tF(x)=c\}|\\
 =&|\{(t,a) \in \mathbb{F}_{p^{n}}\times \mathbb{F}_{p^{n}}\colon  D_tF(a)-D_tF(x)=c\}|\\
 =&N_F(c, x)
 \end{align*}
 where $N_F(c, x)$ represents the number of pairs $(t,a) \in \mathbb{F}_{p^{n}}\times \mathbb{F}_{p^{n}}$ such that $$D_tF(a)-D_tF(x)=c$$ as in Section 2 of \cite{mos1}. Thus we will have the following result concerning the derivative and PGDS parameters.
 
 \begin{thm} \label{diff2-plateaued}
 	Let $F$ be a function from  $\mathbb{F}_{p^{n}}$ to $\mathbb{F}_{p^m}$.  Then the set $G_F$ is a PGDS with parameters $(p^{n+m},p^n ;\alpha, \beta)$ if and only if
 	\[N_F(c, x)=  \begin{cases} 
 	\alpha, & c\ne 0 \\
 	\beta, & c = 0 \\
 	\end{cases}
 	\]
 	for all $x \in \mathbb{F}_{p^n}$ and some constants $\alpha$ and $\beta$.
 \end{thm}

 \begin{remark} 
 	Using Theorem \ref{diff2-plateaued}, Proposition \ref{diff-plateauedprop} and Theorem \ref{plateaued-PGDS}, we are able to prove Theorem 8i. in \cite{mos1} with a different approach using the properties of partial geometric difference sets. This gives an interesting relation between the parameters of a PGDS and the second order derivatives of (vectorial) plateaued functions.
 \end{remark}

 \subsection{A family of vectorial $s-$plateaued functions}
 In this section, we will discuss the link between vectorial $s$-plateaued functions $F(x)=x^d$ from $\mathbb{F}_{p^{n}}$ to $\mathbb{F}_{p^{n}}$ and the cross-correlation function between two $p$-ary $m$-sequences that differ by a decimation $d$. An $m$-sequence and its decimation is defined by $u(t)=\sigma^t$ and $v(t)=u(dt)$ where $\sigma$ is a primitive element of the finite field. The cross-correlation between the sequences $u$ and $v$ is defined by 
 $$\theta(\tau)=\sum_{t=0}^{p^n-2}\zeta_p^{u(t+\tau)-v(t)}.$$ It can be shown that 
 $$\theta(\tau)=-1+\sum_{x \in \mathbb{F}_{p^{n}} }\zeta_p^{Tr_n(ax+x^d)}$$ where $a=-\sigma^\tau$.
 Therefore for $F_1(x)=Tr(x^d)$, we have
 \begin{equation}\label{CC-Walsh}
 \theta(\tau)=-1+\widehat{F_1}(-a).
 \end{equation}
 \begin{thm}\label{cross-plateaued}
 	Let $F(x)=x^{d}$ be a vectorial function from $\mathbb{F}_{p^{n}}$ to $\mathbb{F}_{p^{n}}$ with  $gcd(d,p^n-1)=1$. If the cross-correlation of the $p$-ary $m$-sequences that differ by decimation $d$ takes three values, namely $-1, -1+p^{\frac{n+s}{2}}$ and $-1-p^{\frac{n+s}{2}}$, then $F$ is a vectorial $s$-plateaued function.
 \end{thm}

 \begin{proof}
 	For each $b\in \mathbb{F}_{p^{n}}^*$, we denote by $F_b(x)$ the function $F_b(x)=Tr(bF(x))=Tr(bx^d)$.
 	The walsh transform
 	\[\widehat{F_b}(0)=\sum_{x \in \mathbb{F}_{p^{n}}} \zeta_p^{Tr_n(bx^d)}=0\]
 	since $x^d$ is a permutation.
 	For each $a\in \mathbb{F}_{p^{n}}^*$, the Walsh transform of $F_b(x)$ 
 	\begin{align*}
 	\widehat{F_b}(a)&=\sum_{x \in \mathbb{F}_{p^{n}}} \zeta_p^{Tr_n(bx^d-ax)}\\
 	&=\sum_{x \in\mathbb{F}_{p^{n}}  }\zeta_p^{Tr_n(c^dx^d-\frac{a}{c}cx)}\\
 	&=\sum_{y \in\mathbb{F}_{p^{n}}  }\zeta_p^{Tr_n(y^d-\mu y)}\\
 	&=\widehat{F_1}(\mu)
 	\end{align*}
 	where $b=c^d$ and $\mu=a/c$. Note that any $b\in \mathbb{F}_{p^{n}}^*$ can be written as $b=c^d$ for some $c \in \mathbb{F}_{p^{n}}^*$.
 	Then using Equation (\ref{CC-Walsh}), we immediately obtain the result that $F(x)$ is vectorial s-plateaued.
 	
 \end{proof}
 
 \begin{lemma}
 	Let $p$ be an odd prime and $n,k$ be positive integers with $gcd(n,k)=s$. If $ n/s$ is odd then 
 	$\gcd{(p^n-1,d)}=1$ for $d=(p^{2k}+1)/2$ and $d=p^{2k}-p^k+1$.
 \end{lemma}

 \begin{remark}
 	There are only finitely many known functions with a three-valued cross-correlation. Trachtenberg proved the following in his thesis \cite{trachtenberg}. Let $n$ be an odd integer and $k$ be an integer such that $gcd(n,k)=s$. Then for each of the decimations $d=\frac{p^{2k}+1}{2}$ and $d=p^{2k}-p^k+1$ the cross-correlation function $\theta_d(\tau)$ takes the values $-1,-1\pm p^{\frac{n+s}{2}}$. 
 	
 	This result is generalized later by Helleseth in Theorem 4.9 in \cite{helleseth}. He showed that if $gcd(n,k)=s$ and $ n/s$ is odd then for the same decimations, $\theta_d(\tau)$ has the values $-1,-1\pm p^{\frac{n+s}{2}}$. Our result implies that the corresponding vectorial functions are $s-$plateaued. 
 \end{remark}	
 
 Pott et. al. provided a classification of weakly regular bent functions via partial difference sets, \cite{tan2010}. In this classification authors showed that a function from $\mathbb{F}_3^n$ to $\mathbb{F}_3$ when $n$ is even satisfying $f(-x) = f(x)$ and $f(0)=0$ is weakly regular if and only if a $D_1=\{x: f(x)=1\}$ and $D_2=\{x: f(x)=2\}$ are partial difference sets.  Later in \cite{olmez3}, the author provided a similar classification for  weakly regular bent functions from $\mathbb{F}_3^n$ to $\mathbb{F}_3$ when $n$ is odd via partial geometric difference sets. 
 Next we will also show that our vectorial functions have a similar classification in the case of $p=3$ and $d=\frac{3^{2k}+1}{2}$. 
 
 \begin{thm}\label{partitioned}.
 	Let $n \ge 3$ be an integer and $d=\frac{3^{2k}+1}{2}$ with $\gcd{(n,k)}=s$ and $n/s$ is odd. For $i=0,1,2$ the sets $D_i=\{x: F(x)=Tr_n(x^d)=i\}$ are $(3^n, 3^{n-1}, 3^{2n-3}-3^{n-2}, 3^{n-1}+3^{2n-3}-3^{n-2})$ partial geometric difference sets in the additive group of $\mathbb{F}_{3^{n}}$.
 \end{thm}
  
 \begin{proof}
 	For each $a\in \F_{3^n}^*$,
 	Suppose that $\chi_a(D_1)=x_a+y_a \zeta_3$ with $x_a,y_a\in \mathbb{R}$. Actually, when we are calculating $\chi_a(D_1)$, we are summing powers of $\zeta_3$ as many times as the number of elements in $D_1$ for which $Tr_{n}(ax)=0,1$ or $2$. In other words, if we set 
 	\[D_{1,i}=\{x\in D_1:Tr(x)=i\}, i=0,1,2\]
 	we have 
 	\[\chi_a(D_1)=|D_{1,0}|+|D_{1,1}|\zeta_3+|D_{1,2}|\zeta_3^2=(|D_{1,0}|-|D_{1,2}|)+(|D_{1,1}|-|D_{1,2}|)\zeta_3\]
 	and hence
 	\[x_a=|D_{1,0}|-|D_{1,2}|,y_a=|D_{1,1}|-|D_{1,2}|\] are both in $\mathbb{Z}$.
 	Also note that for any $x \in \mathbb{F}_{3^n}$, 
 	\[F_1(-x)=Tr((-x)^d)=Tr(-x^d)=-F_1(x)\] 
 	and that gives us 
 	\[x\in D_1 \Longleftrightarrow 2x \in D_2.\]
 	As a consequence,
 	$\chi_a(D_2)=\overline{\chi_a(D_1)}=x_a+y_a \zeta_3^2$. 
 	Since $D_i$'s form a partition of the additive group of $\mathbb{F}_{3^{n}}$ 
 	$$\chi_a(D_0)+ \chi_a(D_1)+\chi_a(D_2)=0,$$ and we obtain
 	\[\chi_a(D_0)=y_a-2x_a\]
 	Consider the Walsh transform values $\widehat{F_1}(a), \widehat{F_1}(-a)$ of $F_1(x)=Tr_n(x^d)$.
 	\begin{align*}
 	\widehat{F_1}(-a)&=\sum_{x\in \mathbb{F}_{3^{n}}}\zeta_3^{Tr_{n}(ax+x^d)}\\
 	&=\chi_a(D_0)+\zeta_3\chi_a(D_1)+\zeta_3^2\chi_a(D_2)\\
 	&=y_a-2x_a+\zeta_3(x_a+y_a \zeta_3)+\zeta_3^2(x_a+y_a \zeta_3^2)\\
 	&=y_a-2x_a+\zeta_3(x_a+y_a)+\zeta_3^2(x_a+y_a)=-3x_a\\
 	\end{align*}
 	and 
 	\begin{align*}
 	\widehat{F_1}(a)&=\sum_{x\in \mathbb{F}_{3^{n}}}\zeta_3^{Tr_{n}(-ax+x^d)}\\
 	&=\chi_{-a}(D_0)+\zeta_3\chi_{-a}(D_1)+\zeta_3^2\chi_{-a}(D_2)\\
 	&=\overline{\chi_a(D_0)}+\zeta_3\overline{\chi_a(D_1)}+\zeta_3^2\overline{\chi_a(D_2)}\\
 	&=\overline{\chi_a(D_0)}+\zeta_3\chi_a(D_2)+\zeta_3^2\chi_a(D_1)\\
 	&=y_a-2x_a+\zeta_3(x_a+y_a \zeta_3^2)+\zeta_3^2(x_a+y_a \zeta_3)\\
 	&=3(y_a-x_a).
 	\end{align*}
 	Since $\widehat{F_1}(a),\widehat{F_1}(-a)\in \{0,\pm3^{(n+s)/2}\}$, 
 	\[(x_a,y_a) \in \{(0,0),(0,C),(0,-C),(C,C),(C,0),(C,2C),(-C,-C),(-C,-2C),(-C,0)\}\] 
 	where $C=3^{(n+s-2)/2}$. 
 	In the following table values of $\chi_a(D_0),\chi_a(D_1),\chi_a(D_2), W_{F}(a),W_{F}(-a)$ corresponding to each possible $(x_a,y_a)$ tuple is given.
 	\begin{table}[h]\footnotesize 
 		\centering
 		\begin{tabular} { | c | c | c | c | c | c | c | c | c | c |}
 			\hline 
 			& $(0,0)$ & $(0,C)$ & $(0,-C)$ & $(C,C)$ & $(C,0)$ & $(C,2C)$ & $(-C,-C)$ & $(-C,-2C)$ & $(-C,0)$\\
 			\hline 
 			$\chi_a(D_0)$ & $0$ & $C$ & $-C$ & $-C$ & $-2C$ & $0$ & $C$ & $0$ & $2C$\\
 			\hline
 			$\chi_a(D_1)$ & $0$ & $C\zeta_3$ & $-C\zeta_3$ & $-C\zeta_3^2$ & $C$ & $C\zeta_3-C\zeta_3^2$ & $C\zeta_3^2$ & $C\zeta_3^2-C\zeta_3$ & $-C$\\
 			\hline
 			$\chi_a(D_2)$ & $0$ & $C\zeta_3^2$ & $-C\zeta_3^2$ & $-C\zeta_3$ & $C$ & $C\zeta_3^2-C\zeta_3$ & $C\zeta_3$ & $C\zeta_3-C\zeta_3^2$ & $-C$\\
 			\hline
 			$\widehat{F_1}(a)$ & $0$ & $0$ & $0$ & $-3C$ & $-3C$ & $-3C$ & $3C$ & $3C$ & $3C$\\
 			\hline
 			$\widehat{F_1}(-a)$ & $0$ & $-3C$ & $3C$ & $-0$ & $-3C$ & $3C$ & $0$ & $-3C$ & $3C$\\
 			\hline
 		\end{tabular}
 	\end{table}
 	
 	The tuples $(C,2C)$ and $(-C,-2C)$ are impossible since 
 	\[|\chi_a(D_1)|=\sqrt{x_a^2-x_ay_a+y_a^2}=\sqrt{3C^2}=\sqrt{3^{n+s-1}}\]
 	which contradicts the fact that $|\chi_a(D_1)|\in \mathbb{Z}$ since $n+s-1$ is odd.
 	Therefore the sets $D_1,D_2$ are PGDS.  
 	
 	Next we will show that the tuples $(C,0)$ and $(-C,0)$ are also impossible.
 	
 	First note that $\widehat{F_1}(0)=0$ since the function $Tr(x^d)$ is balanced. For a non-zero element $a$ the following holds 
 	\[ 
 	\begin{split}
 	\widehat{F_1}(a)\widehat{F_1}(-a) =& \chi_a(D_0)\chi_a(D_0)+\zeta_p\chi_a(D_0)\chi_a(D_2)+\zeta_p^2(\chi_a(D_0)\chi_a(D_1))\\
 	&+\zeta_p\chi_a(D_0)\chi_a(D_1)+\zeta_p^2\chi_a(D_1)\chi_a(D_2)+(\chi_a(D_1)\chi_a(D_1))\\
 	&+\zeta_p^2(\chi_a(D_0)\chi_a(D_2))+(\chi_a(D_2)\chi_a(D_2)+\zeta_p\chi_a(D_1)\chi_a(D_2)\\
 	&=\left( \chi_a(D_0)\chi_a(D_0)+ \chi_a(D_0)\chi_a(D_1)+ \chi_a(D_1)\chi_a(D_1)\right) (2-\zeta_p-\zeta_p^2)\\
 	&=3\left( \chi_a(D_0)\chi_a(D_0)+ \chi_a(D_0)\chi_a(D_1)+ \chi_a(D_1)\chi_a(D_1)\right)\\
 	&=3\left( \chi_a(D_0)\chi_a(D_0)- \chi_a(D_1)\chi_a(D_2)\right)\\
 	\end{split}
 	\]
 	We need  the following auxiliary lemma.
 	\begin{lemma} Let $S$ be a $k-$subset of an abelian group $G$ of order $v$. Then
 		$$\displaystyle \sum_{i=0}^{v-1}\chi_i(SS^{-1})= \sum_{i=0}^{v-1}\chi_i(ke+\sum_{g\in G-{e}}a_gg)=vk.$$
 	\end{lemma}
 	
 	\begin{proof} In the group ring $\mathbb{Z}G$ the product $SS^{-1}=k\cdot e+\sum_{g\in G-{e}}a_g\cdot g$ where $a_g \in \mathbb{Z}$. Then 
 		\[ 
 		\begin{split}
 		\displaystyle \sum_{i=0}^{v-1}\chi_i(SS^{-1}) &= \sum_{i=0}^{v-1}\chi_i(ke+\sum_{g\in G-{e}}a_gg)\\
 		&= \sum_{i=0}^{v-1}k\cdot \chi_i(e)+\sum_{g\in G-{e}}a_g\cdot  \sum_{i=0}^{v-1}\chi_i(g)\\
 		&=vk
 		\end{split}
 		\]
 		This holds since for any $g\in G-{e}$ we have $\sum_{i=0}^{v-1}\chi_i(g)=0$ and  $\sum_{i=0}^{v-1}\chi_i(e)=v.$ 
 		
 	\end{proof}
 	Using the previous lemma for $S=D_0,D_1$ separately, we obtain
 	\[ 
 	\begin{split}
 	\sum_{a\in \mathbb{F}_{3^n}}\widehat{F_1}(a)\widehat{F_1}(-a)&=3\sum_{a\in \mathbb{F}_{3^n}}\left( \chi_a(D_0)\chi_a(D_0)- \chi_a(D_1)\chi_a(D_2)\right)\\
 	&=3\sum_{a\in \mathbb{F}_{3^n}} \chi_a(D_0D_0^{-1})- 3\sum_{a\in \mathbb{F}_{3^n}}\chi_a(D_1D_1^{-1})\\
 	&=3\cdot 3^{n}\cdot 3^{n-1}-3\cdot 3^{n}\cdot 3^{n-1}\\
 	&=0
 	\end{split}
 	\]
 	On the other hand, using the values from the table, we can also write
 	\[\sum_{a\in \mathbb{F}_{3^n}}\widehat{F_1}(a)\widehat{F_1}(-a)= 9(\Lambda_1 +\Lambda_2)C^2\]
 	where $\Lambda_1=|\{a \in \F_{3^n}^*:(x_a,y_a)=(C,0)\}|, \Lambda_2=|\{a \in \F_{3^n}^*:(x_a,y_a)=(-C,0)\}|$.
 	This implies that for any $a\in \F_{3^n}^*$, $(x_a,y_a) \ne (C,0), (-C,0)$ and hence $D_0$ is also a PGDS.
 \end{proof}
 
 \begin{remark}
 	By  Theorem \ref{cross-plateaued} we have PGDS with parameters $(v=p^{2n}, k=p^n, \alpha=p^{n}-p^s, \beta=p^{s+n}+p^{n}-p^s)$. If $p=3$ then  by Theorem \ref{partitioned} we also have PGDS with parameters $(v=3^{n}, k=3^{n-1}, \alpha=3^{2n-3}-3^{n-2}, \beta=3^{n-1}+3^{2n-3}-3^{n-2})$. Here we also note that not all decimation will lead such a partition of the finite field $ \F_{3^n}$. For instance, let $D_i=\{x: F(x)=Tr_n(x^d)=i\}$ for $d=3^{2}-3+1$. Then computational results imply that none of the $D_i$'s is a partial geometric difference set in $\F_{3^5}$. In general, it is a challenging task to characterize all functions which can be used to obtain a partition of a group into partial geometric difference sets.  
 \end{remark}
 
 If there is such a partition we can define a set of complex vectors
 $$z_{a}=(\chi_a(D_0),\chi_a(D_1),\chi_a(D_2))$$  for any $a\in \mathbb{F}_{3^n}$. Let $$e=(1,\zeta_3, \zeta_3^2).$$ Then norm of the complex inner product of any vector $z_{a}$ and $e$ is either $0$ or $C$ for some integer $C$.

 \begin{thm}
 	Let $D_0$, $D_1$ and $D_2$ be a partition of $\mathbb{F}_{3^n}$ and $\lambda=3^{\frac{n+s-1}{2}}$ be an integer. 
 	Suppose for $i=0,1,2$ each $D_i$ is a partial geometric difference set such that  $\chi_a(D_i) \in  \{0, \pm \lambda,\pm \lambda \zeta_3,\pm \lambda \zeta_3^2\}$ for each nonprincipal character $\chi_a$. If one of the cases holds
 	\begin{itemize}
 		\item $|D_0|=|D_1|=|D_2|$,
 		\item $|D_i|=|D_j|=3^{n-1}-3^{\frac{n+s-2}{2}}$ and $|D_k|=3^{n-1}+2.3^{\frac{n+s-2}{2}}$,
 		\item $|D_i|=|D_j|=3^{n-1}+3^{\frac{n+s-2}{2}}$ and $|D_k|=3^{n-1}-2.3^{\frac{n+s-2}{2}}$.
 	\end{itemize}
 	and $|<z_a,e>|^2$ is either $0$ or $3\lambda^2$ then 
 	\[ f(x)=\begin{cases} 
 	0, & x \in D_0\\
 	1, & x\in D_1\\
 	2, & x\in D_2\\
 	\end{cases}
 	\]
 	is an $s$-plateaued function.
 \end{thm}
 \begin{proof}
 The Walsh transform
  	\[\widehat{f}(0)=\sum_{x \in \mathbb{F}_{p^{n}}} \zeta_3^{f(x)}=|D_0|+\zeta_3|D_1|+\zeta_3^2|D_2|.\]
  	With the conditions given in the theorem, we obtain
  	\[\widehat{f}(0)=\begin{cases}
  	0 & \text{if}\; |D_0|=|D_1|=|D_2|,\\
  	3^{\frac{n+s}{2}} & \text{if}\;|D_i|=|D_j|=3^{n-1}-3^{\frac{n+s-2}{2}}, |D_k|=3^{n-1}+2.3^{\frac{n+s-2}{2}},\\
  	-3^{\frac{n+s}{2}}  & \text{if}\; |D_i|=|D_j|=3^{n-1}+3^{\frac{n+s-2}{2}}, |D_k|=3^{n-1}-2.3^{\frac{n+s-2}{2}}.
  	\end{cases}\]
  	For each $a\in \mathbb{F}_{3^n}^*$, the Walsh transform of $f(x)$ is 
  	\begin{align*}
  	\widehat{f}(a)&=\sum_{x \in \mathbb{F}_{p^{n}}} \zeta_3^{f(x)-Tr_n(ax)}\\
  	&=\sum_{x \in D_0} \zeta_3^{Tr_n(-ax)} + \zeta_3\sum_{x \in D_1} \zeta_3^{Tr_n(-ax)} + \zeta_3^2\sum_{x \in D_2} \zeta_3^{Tr_n(-ax)}\\
  	&=\chi_{-a}(D_0)+\zeta_3 \chi_{-a}(D_1) + \zeta_3^2\chi_{-a}(D_2).
  	\end{align*}
  	Then with the assumptions of the theorem and after some easy calculations one can show that $|\widehat{f}(a)|^2 \in \{0,3\lambda^2\}$.
 \end{proof}
 
 \section{Results on p-ary Functions}
 
 In this section we will develop some tools to characterize $p-$ary $s-$plateaued functions. Our results are mimicking the results of \cite{olmez2}.
 
 \begin{lemma} \label{Matrix-Eq} Let $f$ be a function from $\mathbb{F}_{p^{n}}$ to $\mathbb{F}_{p}$ and $M=(m_{x,y})$ be a $p^n \times p^n$ matrix where $m_{x,y}=\zeta_p^{f(x+y)}$. Then, $f$ is an $s$-plateaued function  if and only if \begin{equation}MM^*M=p^{n+s}M \label{Mfequation}\end{equation}
 	where $M^*$ is the adjoint of the matrix $M$.
 \end{lemma}
 \begin{proof} Suppose $f$ is an plateaued function with $|\widehat{f}(x)| \in \{0, p^{(n+s)/2}\}$ for all $x \in  \mathbb{F}_{p^{n}}$. Then,
 	\[
 	\begin{split}
 	(MM^*M)_{x,y}&=\sum_{z\in\mathbb{F}_{p^{n}}} \left(\sum_{c\in \mathbb{F}_{p^{n}}} m_{x,c}\overline{m_{z,c}}\right)m_{z,y} \\
 	&=\sum_{z\in \mathbb{F}_{p^{n}}} \left(\sum_{c\in \mathbb{F}_{p^{n}}} F(x+c)\overline{F(z+c)}\right)F(z+y) \\
 	&=\sum_{c\in \mathbb{F}_{p^{n}}} F(x+c)\left(\sum_{w\in \mathbb{F}_{p^{n}} } \overline{F(w)}F(w+y-c)\right)\\
 	&=\sum_{c\in \mathbb{F}_{p^{n}}} (\overline{F}*F)(c-y)F(x+c) \\
 	&=\sum_{u\in \mathbb{F}_{p^{n}}} (F*\overline{F})(u-x-y)F(u) \\
 	&=((F*\overline{F})*F)(x+y).
 	\end{split}
 	\]
 	Let $A=(F*\overline{F})*F$. Then, the Fourier transform of $A$ is
 	$\widehat{A}=\widehat{F}\cdot\widehat{\overline{F}}\cdot\widehat{F}$. Now by Fourier inversion
 	\[
 	\begin{split}
 	A(x+y)&=\frac{1}{p^{n}}\sum_{\beta \in \mathbb{F}_{p^{n}}} \widehat{F}(\beta)\widehat{\overline{F}}(\beta)\widehat{F}(\beta)\zeta_p^{Tr((x+y)\beta)}\\
 	&=p^{n+s}\frac{1}{p^{n}}\sum_{\beta \in \mathbb{F}_{p^{n}}} \widehat{F}(\beta)\zeta_p^{Tr((x+y) \beta)}\\
 	&=p^{n+s}F(x+y).\\
 	\end{split}
 	\]
 	Hence the equation holds.
 	
 	Suppose $MM^*M=p^{n+s}M$. This implies $((F*\overline{F})*F)(x)=p^{n+s}F(x)$ for all $x \in \mathbb{F}_{p^n}$. Apply the Fourier transform on both of the sides. Then,
 	$$\widehat{F}(x)(\widehat{F}(x)\cdot \overline{\widehat{F}(x)}-p^{n+s})=0$$ for all $x \in \mathbb{F}_{p^n}$. Hence, $|\widehat{F}(x)| \in \{0, p^{(n+s)/2}\}$ for all $x \in  \mathbb{F}_{p^{n}}.$
 	
 \end{proof}
 
 \begin{remark} An $n \times n$ complex matrix $M$ is called a Butson-Hadamard matrix if
 	$$MM^* =nI_n.$$ It is easy to see that a $q^n \times q^n $ Butson-Hadamard matrix $M$ also satisfies 
 	$$MM^*M=q^nM.$$ Our result implies that $M$ can be associated with $0-$plateaued function. This indicates the well-known connection between Butson-Hadamard matrices and bent functions.  
 \end{remark}
 
 As a corollary of Lemma \ref{Matrix-Eq}, we can characterize $s$-plateaued functions with their first and  second derivatives. First and second derivative of a $p-$ary function is defined by 
 $$D_af(x)=f(x+a)-f(x)$$ and
 $$D_aD_bf(x)=f(x+a+b)+f(x)-f(x+a)-f(x+b)$$ 
 respectively.
 
 \begin{cor}[Theorem 3, \cite{mos1}] $f$ is an $s$-plateaued function from $\mathbb{F}_{p^{n}}$ to $\mathbb{F}_{p}$ if and only if the expression  $\sum_{a,b\in \mathbb{F}_{p^{n}}} \zeta_p ^{D_aD_b f(u)}$ does not depend on $u \in \mathbb{F}_{p^{n}} $. This constant expression equals to $p^{n+s}$.
 \end{cor}
 \begin{proof}  Since the equation $$MM^*M=p^{n+s}M$$ holds, $$M^*MM^*=p^{n+s}M^*$$ holds too.
 	Fix two non-zero elements $x$ and $y$ of $\mathbb{F}_{p^{n}}$ and let $u=x+y$.
 	\[
 	\begin{split}
 	\sum_{z\in\mathbb{F}_{p^{n}}} \left(\sum_{c\in \mathbb{F}_{p^{n}}} \overline{m_{x,c}}m_{z,c}\right)\overline{m_{z,y}} &=\sum_{z\in \mathbb{F}_{p^{n}}} \left(\sum_{c\in \mathbb{F}_{p^{n}}} \zeta_p^{-f(x+c)}\zeta_p^{f(z+c)}\right)\zeta_p^{-f(z+y)} \\
 	&=\sum_{c,z\in \mathbb{F}_{p^{n}}} \zeta_p ^{-f(x+c)+f(z+c)-f(z+y)}\\
 	&=p^{n+s}\zeta_p ^{-f(x+y)}
 	\end{split}
 	\]	
 	Now let $z=a+x$ and $c=b+y$. Then
 	$$p^{n+s}=\sum_{a,b\in \mathbb{F}_{p^{n}}} \zeta_p ^{-f(x+y+b)+f(a+b+x+y)-f(a+x+y)+f(x+y)}$$
 	holds.
 	Thus, $$p^{n+s}=\sum_{a,b\in \mathbb{F}_{p^{n}}} \zeta_p ^{D_aD_b f(u)}.$$
 \end{proof}
 
 \begin{cor} Let $\Delta_f(a)=\sum_{x \in \mathbb{F}_{p^n}} \zeta_p^{D_af(x)}$. $f$ is an $s$-plateaued function from $\mathbb{F}_{p^{n}}$ to $\mathbb{F}_{p}$ if and only if  $\sum_{a\in \mathbb{F}_{p^{n}}} \overline{\Delta_f(a)}\Delta_f(a)=p^{2n+s}$. 
 \end{cor}
 \begin{proof}  We have  $$MM^*MM^*=p^{n+s}MM^*.$$Let  $$N=MM^*.$$ Then
 	\[
 	\begin{split}
 	(N)_{0,a}&=\sum_{x\in \mathbb{F}_{p^{n}}} \zeta_p^{f(x)-f(a+x)} \\
 	&=\overline{\Delta_f(a)}\\
 	\end{split}
 	\]	
 	and
 	\[
 	\begin{split}
 	(N)_{a,0}&=\sum_{x\in \mathbb{F}_{p^{n}}} \zeta_p^{f(a+x)-f(x)} \\
 	&=\Delta_f(a)\\
 	\end{split}
 	\]	
 	Therefore
 	\[
 	\begin{split}
 	(N^2)_{0,0}&=\sum_{a\in \mathbb{F}_{p^{n}}} \overline{\Delta_f(a)}\Delta_f(a)\\
 	&=p^{n+s}(N)_{0,0}\\
 	&=p^{2n+s}\\
 	\end{split}
 	\]	
 \end{proof}

 Now we will use our characterization to provide a simple construction of $s$-plateaued functions. If $A$ is an $m \times n$ matrix and $B$ is an $s \times t$ matrix, then the Kronecker product $A \otimes B$ is the $ms \times nt$ block matrix:
 
 $${\displaystyle \mathbf {A} \otimes \mathbf {B} ={\begin{bmatrix}a_{11}\mathbf {B} &\cdots &a_{1n}\mathbf {B} \\\vdots &\ddots &\vdots \\a_{m1}\mathbf {B} &\cdots &a_{mn}\mathbf {B} \end{bmatrix}},}$$
 
 \begin{prop}
 	Let $f: \mathbb{F}_{p^{n}} \to \mathbb{F}_{p} $ and $g : \mathbb{F}_{p^{m}} \to \mathbb{F}_{p}$ be $s_1$-plateaued and $s_2$-plateaued functions respectively. Let $M$ and $N$ be the matrices whose entries determined by  $m_{x,y}=\zeta_p^{f(x+y)}$ and $n_{a,b}=\zeta_p^{g(a+b)}$. Let $P=M \otimes N$ be the Kronecker product of $M$ and $N$. Then
 	$$PP^*P=p^{n+m+s_1+s_2}P$$ holds.
 \end{prop}
 
 \begin{proof} 
 	Let $P=M \otimes N$. Then
 	\[
 	\begin{split}
 	PP^*P&=\left( M \otimes N\right) \left( M^* \otimes N^*\right)\left( M \otimes N\right)   \\
 	&=\left( MM^*M\right) \otimes \left( NN^*N \right)\\
 	&=\left( p^{n+s_1}M\right) \otimes \left( p^{m+s_2}N \right)\\
 	&=p^{n+m+s_1+s_2} \left( M\otimes N\right) \\
 	&=p^{n+m+s_1+s_2}P\\
 	\end{split}
 	\]		
 \end{proof}
 
 \begin{cor}
 	Let $f: \mathbb{F}_{p^{n}} \to \mathbb{F}_{p} $ and $g : \mathbb{F}_{p^{m}} \to \mathbb{F}_{p}$ be two $s_1$-plateaued and $s_2$-plateaued functions respectively. Then there exists an $(s_1+s_2)$-plateaued function from $\mathbb{F}_{p^{n+m}}$ to $\mathbb{F}_{p}$.
 \end{cor}
 \begin{proof} Let $M$ and $N$ be matrices associated with $f$ and $g$ such that $m_{x,y}=\zeta_p^{f(x+y)}$ and $n_{a,b}=\zeta_p^{g(a+b)}$.	Let $P=M \otimes N$. We need to show that there exist a function $h$ such that the entries of $P$ can be associated with $h$.
 	
 	We will index the rows and columns of $P$ by the elements of $\mathbb{F}_{p^{n+m}}$. First note that there is a subgroup $H$ of the additive group of $\mathbb{F}_{p^{n+m}}$ which is isomorphic to the additive group of $\mathbb{F}_{p^{m}}$. Let us fix a transversal $T=\{\gamma_1,\gamma_2,\dots, \gamma_{p^n} \}$ of this subgroup in $\mathbb{F}_{p^{n+m}}$. Now order the rows and columns of $P$ by $\gamma_1+H$, $\gamma_2+H$, ..., $\gamma_{p^n}+H$ in the block form. Here we have isomorphisms namely $$\phi_1: H \to \mathbb{F}_{p^{m}}$$ and $$\phi_2: \{\gamma_1+H,\gamma_2+H,\dots, \gamma_{p^n}+H \} \to \mathbb{F}_{p^{n}}$$
 	
 	If $x,y \in \mathbb{F}_{p^{n+m}}$ then $x=\gamma_i+u$ and $y=\gamma_j+v$ for unique elements $u,v \in H$. Now let us examine the $xy$-th entry of $P$.
 	$$P_{x,y}=\zeta_p^{f(\phi_2(\gamma_i+H)+\phi_2(\gamma_j+H))} \cdot \zeta_p^{g(\phi_1(u)+\phi_1(v))}= \zeta_p^{h(x+y)}$$ where
 	$h$ is the desired $(s_1+s_2)$-plateaued function from $\mathbb{F}_{p^{n+m}}$ to $\mathbb{F}_{p}.$ Moreover if $y=0$ then $\gamma_j=v=0$. Thus
 	$$h(x)= f(\phi_2(\gamma_i+H))+g(\phi_1(u)).$$
 	
 \end{proof}
 
 %\oktay{\textbf{if f is a bent and g is a linear function, what can we get from the above corrollary?}}
 
 %
 %Next we will characterize plateaued functions by their associated difference sets. This technique has been effectively used for bent %functions but this problem is an open problem for plateaued functions. 
 %
 \subsection{Partially bent functions }
 
 This section is devoted to investigate a family of  $s-$plateaued functions known as partially bent functions. A $p-$ary function is called partially bent if the derivative $D_af$ is either balanced or constant for any $a \in \mathbb{F}_{p^n}$. Here we will provide some characterization via their associated designs.
 
 Let $f$ be an $s$-plateaued function and for $a \in \mathbb{F}_{p^{n}} $ define the set
 $$T_a=\{x+a,f(x)+f(a): x \in \mathbb{F}_{p^{n}} \}.$$
 Fix an element $a$ of $\mathbb{F}_{p^{n}} $, then  the graph of $f$ can be also written as
 $$G_f=\{(x,f(x)): x \in \mathbb{F}_{p^{n}}\}=\{(x+a,f(x+a)): x \in \mathbb{F}_{p^{n}}\}.$$
 \begin{lemma} Let  $f$ be an $s$-plateaued function with $f(0)=0$.
 	If $f$ has a linear structure  $\Lambda$ then $T_a=G_f$ for all $a \in \Lambda$.
 \end{lemma}
 
 \begin{cor} Let  $f$ be an $s$-plateaued function with $f(0)=0$ and linear structure $\Lambda$ of dimension $m$. Then the incidence matrix $A$ of the design associated with the partial geometric difference set $G_f$ can be written as a Kronecker product of $1\times p^m$ all-ones matrix $j$ and an incidence matrix $N$ of a partial geometric design.  
 \end{cor}
 \begin{proof} Let $j$ be the $1\times p^m$ all-ones matrix.  Let $D$ be the block design associated with $G_f$ where the point set is $\mathbb{F}_{p^{n}} \times \mathbb{F}_{p} $ and the blocks are the translates of the graph of $f$. Suppose $\textbf{B}$ is a block in $D$. Note that $\textbf{B}=(u,v)+G_f$ for some $(u,v) \in \mathbb{F}_{p^{n}} \times \mathbb{F}_{p}.$ Then for each $a\in \Lambda$ we have 
 	$$\textbf{B}+(a,f(a))=\textbf{B}$$ since   
 	$$(u,v)+G_f=\{(x+u+a,f(x+a)+v): x \in \mathbb{F}_{p^{n}}\}.$$ Thus each block is repeated $p^m$ many times.  Therefore, $A=j\otimes N$ for some incidence matrix $N$. Now we are going to show that $N$ is an incidence matrix of a partial geometric design. Since $A$ is an incidence matrix of a partial geometric design,
 	\[	\begin{split}
 	AA^tA&=(\beta-\alpha) j \otimes N+\alpha J_{p^n \times p^n}\\
 	&=j \otimes (\beta-\alpha) N+\alpha j \otimes J_{p^n \times p^{n-m}}\\
 	&= j\otimes (\beta-\alpha) N+\alpha J_{p^n \times p^{n-m}}.\\
 	\end{split}
 	\]
 	We also have 
 	
 	\[	\begin{split}
 	AA^tA&=jj^tj\otimes NN^tN\\
 	&=p^m j\otimes NN^tN\\
 	&=j\otimes p^m NN^tN.\\
 	\end{split}
 	\]
 	By comparing the left hand sides we can conclude that the equation
 	$$NN^tN=\frac{(\beta-\alpha)}{p^m} N+\frac{\alpha}{p^m} J_{p^n \times p^{n-m}}$$
 	holds.
 \end{proof}
 
 Let $f$ be an $s$-plateaued function from  $\mathbb{F}_{p^{n}}$ to $\mathbb{F}_{p}$.  Then the set $D_f$ is a PGDS with parameters $(p^{n+1},p^n ; p^{2n-1}-p^{n+s-1}, p^{2n-1}-p^{n+s-1}+p^{n+s})$.  If $f$ is a partially bent function then there is an integer $s\geq 0$ such that $f$ is $s$-plateaued and the linear space of $f$ has dimension $s$. Our observation yields the following result.
 
 \begin{cor}
 	If $f$ is partially bent function then $f$ can be associated with a partial geometric design with parameters 
 	$$v=p^{n+1}, b=p^{n+1-s}, k=p^n, r=p^{n-s}, \alpha= p^{2n-1-s}-p^{n-1}, \beta=p^n+p^{2n-1-s}-p^{n-1}.$$
 \end{cor}

\end{document}